\documentclass[a4paper,11pt]{amsart}
\usepackage[a4paper,margin=1in]{geometry}
\usepackage{amsmath,amssymb,amsthm,mathtools}
\usepackage[T1]{fontenc}
\usepackage{lmodern}
\usepackage{microtype}
\usepackage{hyperref}

\hypersetup{colorlinks=true,linkcolor=blue,citecolor=blue,urlcolor=blue,pdftitle={Bene\v{s} and Shuffle-Exchange Counterexamples},pdfauthor={Przemek Chojecki}}

\newtheorem{theorem}{Theorem}
\newtheorem{lemma}[theorem]{Lemma}
\newtheorem{corollary}[theorem]{Corollary}
\newtheorem{proposition}[theorem]{Proposition}
\theoremstyle{remark}
\newtheorem{remark}[theorem]{Remark}

\newcommand{\Sym}{\operatorname{Sym}}
\newcommand{\calG}{\mathcal G}

\newcommand{\Z}{\mathbb Z}

\title{Bene\v{s} and Shuffle-Exchange Counterexamples}
\author[Przemek Chojecki]{Przemek Chojecki\\{\normalfont\small\lowercase{ulam.ai}}}
\date{July 1, 2026}

\begin{document}

\begin{abstract}
We give explicit counterexamples to two rearrangeability conjectures for shuffle-type networks. First, for every $N\ge2$ we construct a simple $N$-regular ordered two-stage graph $L_N$ with $F(L_N)=2$ and $R(L_N)\ge N$, refuting the graph-theoretic Bene\v{s} inequality $R(L)\le2F(L)$ and its partition-stabilizer form as stated on Open Problem Garden. We retain the sharp cut obstruction, exact mask-composition identity, exact middle criterion, first nontrivial-level result, and balanced-middle sufficient condition that explain which extra hypotheses can replace mere external connectivity. Second, for the standard directed shuffle-exchange network, we prove $d(k,3)=6$ for every $k\ge3$, while the known binary value is $d(2,3)=5$. Hence the shuffle-exchange conjecture $d(k,n)=2n-1$ fails already at $(k,n)=(3,3)$, and the remaining upper bound $d(k,n)\le3n-3$ for $k\ge3$ suggests $d(k,n)=3n-3$ as a natural replacement problem.
\end{abstract}
\maketitle

\section{The graph counterexample}

We use directed notation for ordered two-stage graphs. Thus an arrow $a\to b$ means an edge from the left copy of $a$ to the right copy of $b$. If $L$ is such a graph, $L^r$ denotes the $r$-fold concatenation with layers $0,1,\ldots,r$. A plain path in $L^r$ is a path from layer $0$ to layer $r$ using exactly one edge in each stage. The graph $L^r$ is externally connected if every source-target pair is joined by a plain path. If $L$ is $k$-regular, a mask for $L^r$ is a two-stage multigraph on the same sources and targets in which every source and every target has degree $k$. The graph $L^r$ is rearrangeable if every mask can be routed by mutually edge-disjoint plain paths. Write $F(L)$ for the least $r$ such that $L^r$ is externally connected, and $R(L)$ for the least $r$ such that $L^r$ is rearrangeable, with value $\infty$ if there is no such $r$. This is the graph-theoretic language for a conjectural extension of Bene\v{s}' rearrangeability problem; see \cite{Benes75,Hwang89,OPGgraph}.

\begin{theorem}\label{thm:graph}
For every integer $N\ge2$ there is a simple $N$-regular ordered two-stage graph $L_N$ such that $F(L_N)=2$ and $R(L_N)\ge N$.
\end{theorem}

\begin{proof}
Let $V=X\sqcup Y$, where $X=\{x_0,x_1,\ldots,x_{N-1}\}$ and $Y=\{y_0,y_1,\ldots,y_{N-1}\}$. Define $L_N$ by the four rules
\[
 x_0\to y_j\ (0\le j<N),\qquad x_i\to x_j\ (1\le i<N,\ 0\le j<N),
\]
\[
 y_0\to x_j\ (0\le j<N),\qquad y_i\to y_j\ (1\le i<N,\ 0\le j<N).
\]
There is at most one edge from any left vertex to any right vertex, so $L_N$ is simple. Each vertex has outdegree $N$. Each $x_j$ receives one edge from $y_0$ and one edge from each of $x_1,\ldots,x_{N-1}$, and each $y_j$ receives one edge from $x_0$ and one edge from each of $y_1,\ldots,y_{N-1}$; hence every indegree is $N$.

The graph $L_N$ is not externally connected in one stage, since $x_0$ has no edge to any $x_j$. It is externally connected in two stages. With $i\ne0$, paths of the forms $x_0\to y_0\to x_j$, $x_0\to y_i\to y_j$, $x_i\to x_0\to y_j$, and $x_i\to x_1\to x_j$ cover all cases with source in $X$, and the cases with source in $Y$ are symmetric. Thus $F(L_N)=2$.

For later use we note that $L_N^r$ is externally connected for every $r\ge2$. The case $r=2$ was just proved. For $r\ge3$, loops at each non-hub vertex allow padding. For example, $x_0$ reaches $x_j$ by $x_0\to y_1$, then loops at $y_1$ as needed, then uses $y_1\to y_0\to x_j$; $x_0$ reaches $y_j$ by looping at $y_1$ and then using $y_1\to y_j$; $x_i$ with $i\ne0$ reaches $x_j$ by looping at $x_1$ and then using $x_1\to x_j$; and $x_i$ with $i\ne0$ reaches $y_j$ by looping at $x_1$ and then using $x_1\to x_0\to y_j$. The $Y$-source cases are symmetric, and the number of loops is chosen to make the length exactly $r$.

It remains to prove the lower bound on $R(L_N)$. In each stage of $L_N^r$ the only edges from $X$ to $Y$ are $x_0\to y_j$ for $0\le j<N$, so across all $r$ stages there are exactly $rN$ such crossing edges. Consider the mask $M$ with $N$ parallel requests $x_i\to y_i$ for each $i$ and $N$ parallel requests $y_i\to x_i$ for each $i$. This is a valid $N$-regular mask. The $X$-to-$Y$ half of $M$ consists of $N^2$ requests, and every corresponding plain path must use at least one $X$-to-$Y$ crossing edge. Since a routing uses mutually edge-disjoint paths, $L_N^r$ cannot route $M$ when $rN<N^2$. Hence $L_N^r$ is not rearrangeable for every $r<N$, and $R(L_N)\ge N$.
\end{proof}

The Open Problem Garden graph formulation of the Bene\v{s} conjecture asks whether, for a simple regular ordered two-stage graph $L$, external connectivity of $L^m$ implies rearrangeability of $L^{2m}$; equivalently, whether $R(L)\le 2F(L)$ \cite{OPGgraph}. Taking $N\ge5$ in Theorem \ref{thm:graph} gives $F(L_N)=2$ and $R(L_N)\ge N>4=2F(L_N)$, so that formulation is false. In fact $F(L_N)$ is fixed while $R(L_N)$ is unbounded, so no upper bound for $R(L)$ can depend only on $F(L)$ in the class of all simple regular ordered two-stage graphs.

\section{Partitions and stabilizers}

Let $\mathbf h_0,\ldots,\mathbf h_r$ be partitions of a finite set $E$. The associated layered graph $G(\mathbf h_0,\ldots,\mathbf h_r)$ has layer $i$ equal to the set of blocks of $\mathbf h_i$; for each $i=1,\ldots,r$ and each $e\in E$, it has one edge labelled $e$ from the block of $\mathbf h_{i-1}$ containing $e$ to the block of $\mathbf h_i$ containing $e$. For a partition $\mathbf h$, let $S(\mathbf h)$ be its setwise block stabilizer in $\Sym(E)$. Permutations act on the left, and products are composed right to left.

\begin{lemma}\label{lem:routing}
Let $\pi\in\Sym(E)$. Form the mask $M(\pi)$ from the blocks of $\mathbf h_0$ to the blocks of $\mathbf h_r$ by giving each $e\in E$ a request from the $\mathbf h_0$-block containing $e$ to the $\mathbf h_r$-block containing $\pi(e)$. Then $M(\pi)$ is routable in $G(\mathbf h_0,\ldots,\mathbf h_r)$ if and only if $\pi\in S(\mathbf h_r)\cdots S(\mathbf h_0)$.
\end{lemma}

\begin{proof}
Suppose first that $\pi=\alpha_r\cdots\alpha_0$ with $\alpha_i\in S(\mathbf h_i)$. For $e\in E$, put $z_i(e)=\alpha_i\cdots\alpha_0(e)$. The stage-$i$ edge labelled $z_{i-1}(e)$ starts in the $\mathbf h_{i-1}$-block of $z_{i-1}(e)$ and ends in the $\mathbf h_i$-block of $z_{i-1}(e)$, which is also the $\mathbf h_i$-block of $z_i(e)$ because $\alpha_i$ fixes each $\mathbf h_i$-block setwise. These edges form a plain path from the $\mathbf h_0$-block of $e$ to the $\mathbf h_r$-block of $\pi(e)$. At each stage the labels $z_{i-1}(e)$ are all distinct as $e$ varies, so the paths are edge-disjoint.

Conversely, suppose $M(\pi)$ is routed. For each $e\in E$, let $w_i(e)$ be the label of the routed edge used between layers $i-1$ and $i$. For fixed $i$ these labels are distinct. Define $\alpha_0(e)=w_1(e)$, define $\alpha_i(w_i(e))=w_{i+1}(e)$ for $1\le i<r$, and define $\alpha_r(w_r(e))=\pi(e)$. These maps are permutations of $E$. The path condition implies that each pair of elements matched by $\alpha_i$ lies in the same block of $\mathbf h_i$, with the analogous statements at the first and last layers; hence $\alpha_i\in S(\mathbf h_i)$. The composition is $\pi=\alpha_r\cdots\alpha_0$.
\end{proof}

When $\mathbf h_0$ and $\mathbf h_r$ are uniform of the same block size $k$, every $k$-regular mask arises as $M(\pi)$ for some permutation $\pi$: split each initial block according to the outgoing demands, and then biject the elements assigned to a final block onto that block. Hence, in this uniform situation, $G(\mathbf h_0,\ldots,\mathbf h_r)$ is rearrangeable if and only if the sequence of partitions is complete, up to reversal of the product order; reversal is harmless for equality with $\Sym(E)$, since the inverse of a full product is full.

The analogous one-path statement gives transitivity. Indeed, a plain path with labels $w_1,\ldots,w_r$ from the $\mathbf h_0$-block of $e$ to the $\mathbf h_r$-block of $f$ can be extended to permutations $\alpha_i\in S(\mathbf h_i)$ with $\alpha_0(e)=w_1$, $\alpha_i(w_i)=w_{i+1}$, and $\alpha_r(w_r)=f$; conversely, the first half of the proof of Lemma \ref{lem:routing} turns such a product into a plain path. Therefore external connectivity is equivalent to transitivity of the same stabilizer product, again up to reversal, since a set of permutations is transitive if and only if its inverse set is transitive.

Now let $L_N$ be the graph from Theorem \ref{thm:graph}, and let $E$ be its edge set. Thus $|E|=2N^2$. For each vertex $v\in V$, put $U_v=\{e\in E:e\text{ leaves }v\}$ and $I_v=\{e\in E:e\text{ enters }v\}$. The partition $\mathbf u=\{U_v:v\in V\}$ is uniform, with $2N$ blocks of size $N$. Since every indegree is $N$, choose bijections $\varphi_v:U_v\to I_v$ and combine them into a permutation $\varphi$ of $E$. Then $\varphi(U_v)=I_v$ for every $v$.

Because $L_N$ is simple, $|U_a\cap I_b|\le1$ for all $a,b\in V$, and hence $\mathbf u\wedge\varphi(\mathbf u)=\mathbf0$. Moreover, the graph between the partitions $\varphi^i(\mathbf u)$ and $\varphi^{i+1}(\mathbf u)$ is canonically isomorphic to $L_N$: applying $\varphi^{-i}$ sends $\varphi^i(U_a)\cap\varphi^{i+1}(U_b)$ to $U_a\cap I_b$.

\begin{corollary}\label{cor:partition}
For every integer $n\ge3$ there exist a finite set $E$, a uniform partition $\mathbf u$ of $E$, and a permutation $\varphi\in\Sym(E)$ such that $\mathbf u\wedge\varphi(\mathbf u)=\mathbf0$ and $(\varphi S(\mathbf u))^n$ is transitive, but $\Sym(E)\ne(\varphi S(\mathbf u))^{2n-1}$.
\end{corollary}

\begin{proof}
Choose $N>2n-2$ and construct $E,\mathbf u,\varphi$ from $L_N$ as above. Put $S_i=S(\varphi^i(\mathbf u))$. Since $\varphi S_0\varphi^{-1}=S_1$, induction gives $(\varphi S_0)^m=S_1S_2\cdots S_m\varphi^m$ for every $m\ge1$. Therefore $(\varphi S_0)^m$ is transitive if and only if $S_1S_2\cdots S_m$ is transitive, equivalently if and only if $S_0S_1\cdots S_{m-1}$ is transitive. Likewise, $(\varphi S_0)^m=\Sym(E)$ if and only if $S_0S_1\cdots S_{m-1}=\Sym(E)$.

Since $L_N^{n-1}$ is externally connected, the partition sequence $\mathbf u,\varphi(\mathbf u),\ldots,\varphi^{n-1}(\mathbf u)$ has a transitive stabilizer product. Hence $(\varphi S(\mathbf u))^n$ is transitive. On the other hand, $L_N^{2n-2}$ cannot route the mask used in the proof of Theorem \ref{thm:graph}, because $2n-2<N$. That mask is $N$-regular, equal to the common block size, so it is $M(\pi)$ for some permutation $\pi$ of $E$. By Lemma \ref{lem:routing}, $\pi$ does not belong to the relevant stabilizer product, and therefore the partition sequence $\mathbf u,\varphi(\mathbf u),\ldots,\varphi^{2n-2}(\mathbf u)$ is not complete. Thus $S_0S_1\cdots S_{2n-2}\ne\Sym(E)$, and $(\varphi S(\mathbf u))^{2n-1}\ne\Sym(E)$.
\end{proof}

The partition formulation on Open Problem Garden states precisely the implication contradicted by Corollary \ref{cor:partition} \cite{OPGpart}; the UnsolvedMath entry OPG-37181 mirrors that problem as an open problem in the OpenGarden set \cite{UM37181}. The construction above therefore refutes those formulations as written. It does not address the special shuffle-exchange conjecture, whose transition graphs have much stronger homogeneity than the bottleneck family used here.

\section{Additional hypotheses}

The examples fail for a structural reason. Rearrangeability forces cut-capacity conditions that external connectivity does not see. The next proposition is both necessary and sharp in the following sense: if a displayed inequality fails, the proof constructs a particular mask that cannot be routed.

\begin{proposition}[Sharp cut obstruction]\label{prop:cut}
Let $G$ be a $k$-regular multistage graph with source set $A_0$ and target set $A_r$, with $|A_0|=|A_r|$. If $G$ is rearrangeable, then every edge cut separating a source subset $P\subseteq A_0$ from a target subset $Q\subseteq A_r$---that is, every set of edges whose deletion leaves no plain path from $P$ to $Q$---has size at least $k\min(|P|,|Q|)$. Conversely, if such a cut has smaller size, then there is a $k$-regular mask not routable in $G$.
\end{proposition}

\begin{proof}
Assume $|P|\le |Q|$. Construct a $k$-regular mask whose $k|P|$ requests out of $P$ all end in $Q$: fill $k|P|$ of the available target slots in $Q$, and then fill all remaining source and target slots arbitrarily, allowing parallel edges. Every routing of this mask must send those $k|P|$ paths across any cut separating $P$ from $Q$, so the cut has size at least $k|P|$. If a smaller cut exists, this same mask is unroutable. The case $|Q|\le |P|$ is identical, using all $k|Q|$ target slots in $Q$ and filling them with requests from $P$.
\end{proof}

For $L_N^r$, take $P=X$ and $Q=Y$. The cut consisting of all $X$-to-$Y$ crossing edges has size $rN$, whereas Proposition \ref{prop:cut} requires $N^2$. This is exactly the obstruction in Theorem \ref{thm:graph}.

\begin{corollary}\label{cor:noexternal}
For every $m\ge2$ and every $q\ge1$ there is a simple regular ordered two-stage graph $L$ such that $L^m$ is externally connected, but $L^q$ is not rearrangeable.
\end{corollary}

\begin{proof}
Choose $N>q$ and take $L=L_N$. Since $m\ge2$, Theorem \ref{thm:graph} gives external connectivity of $L_N^m$, while $R(L_N)\ge N>q$ gives non-rearrangeability of $L_N^q$.
\end{proof}

Thus any corrected general conjecture must add a robust cut-capacity or routing hypothesis, or else restrict to a substantially more homogeneous class of networks. External connectivity alone cannot imply rearrangeability with any bound depending only on the external-connectivity exponent.

Before imposing checkable sufficient hypotheses, it is useful to record the exact algebra of masks. If $G$ is a $k$-regular multistage graph from $A$ to $B$, let $\mathcal R(G)$ be the set of $k$-regular masks from $A$ to $B$ that are routable in $G$, and let $\mathcal M_k(A,B)$ be the set of all $k$-regular masks from $A$ to $B$. If $\mathcal A\subseteq \mathcal M_k(A,B)$ and $\mathcal B\subseteq \mathcal M_k(B,C)$, define $\mathcal A\star\mathcal B$ as follows. A mask $M$ lies in $\mathcal A\star\mathcal B$ if it is obtained from some $P\in\mathcal A$ and $Q\in\mathcal B$ by, for each $b\in B$, pairing bijectively the $k$ requests of $P$ ending at $b$ with the $k$ requests of $Q$ starting at $b$, and replacing each paired pair $a\to b$, $b\to c$ by one request $a\to c$. This operation is the balanced composition of masks through the middle set $B$.

\begin{theorem}[Mask-composition identity]\label{thm:composition}
Let $G$ and $H$ be $k$-regular multistage graphs from $A$ to $B$ and from $B$ to $C$, respectively, with all layers finite and of the same cardinality. Then
\[
 \mathcal R(GH)=\mathcal R(G)\star\mathcal R(H).
\]
\end{theorem}

\begin{proof}
Let $M$ be routed in $GH$, and label the requests of $M$ by their routed paths. Cutting each path at the middle layer $B$ gives a left mask $P$ from $A$ to $B$, a right mask $Q$ from $B$ to $C$, and a pairing at each middle vertex. The masks $P$ and $Q$ are routable in $G$ and $H$. They are also $k$-regular: a $k$-regular mask has $k|A|$ requests, and each stage of a $k$-regular multistage graph with layers of size $|A|$ has $k|A|$ edges, so an edge-disjoint routing uses every edge in every stage; in particular exactly $k$ left halves end at each vertex of $B$ and exactly $k$ right halves start there. Thus $M\in\mathcal R(G)\star\mathcal R(H)$.

Conversely, suppose $M$ is obtained by balanced composition from $P\in\mathcal R(G)$ and $Q\in\mathcal R(H)$. Choose edge-disjoint routings of $P$ and $Q$. For each paired pair of requests $a\to b$ and $b\to c$, concatenate the corresponding left and right paths. The two routings occupy disjoint sets of stages, and each half-routing is edge-disjoint within its own stages, so the concatenated paths are edge-disjoint and route $M$ in $GH$.
\end{proof}

\begin{corollary}[Logically sharp doubled criterion]\label{cor:sharpdoubling}
Let $L$ be a simple $k$-regular ordered two-stage graph on $V$, and let $m\ge1$. Then $L^{2m}$ is rearrangeable if and only if
\[
 \mathcal M_k(V,V)=\mathcal R(L^m)\star\mathcal R(L^m).
\]
Equivalently, the weakest possible additional hypothesis, among hypotheses expressed only by balanced factorization through the middle layer, is that every $k$-regular endpoint mask admit such a factorization into two masks routable in $L^m$.
\end{corollary}

\begin{proof}
The equality is exactly Theorem \ref{thm:composition} with $G=H=L^m$, together with the definition of rearrangeability. If the equality fails, a mask outside $\mathcal R(L^m)\star\mathcal R(L^m)$ is not routable in $L^{2m}$; if it holds, all masks are routable. Hence no strictly weaker middle-factorization condition can imply rearrangeability in this general class.
\end{proof}

The single-mask form of the same criterion is often easier to use. Let $A$ be a $k$-regular multistage graph from a source set $S$ to a middle set $U$, and let $B$ be a $k$-regular multistage graph from $U$ to a target set $T$, all three sets having the same cardinality. For a $k$-regular mask $M$ from $S$ to $T$, a balanced middle assignment is a map $z$ from the request multiset of $M$ to $U$ such that $|z^{-1}(u)|=k$ for every $u\in U$. It produces two $k$-regular masks: the left mask $M_A$ with requests $s(e)\to z(e)$ and the right mask $M_B$ with requests $z(e)\to t(e)$.

\begin{theorem}[Exact middle criterion]\label{thm:exactmiddle}
The concatenation $AB$ is rearrangeable if and only if every $k$-regular mask $M$ from $S$ to $T$ has a balanced middle assignment $z$ for which $M_A$ is routable in $A$ and $M_B$ is routable in $B$.
\end{theorem}

\begin{proof}
This is Theorem \ref{thm:composition} with the balanced composition written request by request. A routing in $AB$ supplies $z(e)$ by recording where the path for $e$ meets the middle layer. Conversely, routings of $M_A$ and $M_B$ concatenate along the assigned middle vertices.
\end{proof}

Theorem \ref{thm:exactmiddle} is the minimal statement at the level of a fixed middle cut: it is necessary and sufficient. Corollary \ref{cor:sharpdoubling} is the corresponding logically sharp doubled form. The following sufficient criterion is stronger, but useful because it can be verified without knowing the target mask in advance. It isolates the usual Bene\v{s} mechanism: color a regular mask into perfect matchings, route each color class through a balanced part of the middle layer, and concatenate the two half-routings.

\begin{lemma}\label{lem:konig}
Every $k$-regular bipartite multigraph with the same number of vertices on both sides decomposes into $k$ perfect matchings.
\end{lemma}

\begin{proof}
Hall's condition holds: if $A$ is a set of left vertices, the $k|A|$ edges incident with $A$ all end in its neighbor set, whose vertices have degree at most $k$, so $|N(A)|\ge |A|$. Thus there is a perfect matching. Removing it leaves a $(k-1)$-regular bipartite multigraph, and induction completes the proof.
\end{proof}

\begin{proposition}[First nontrivial level]\label{prop:firstlevel}
Let $L$ be a simple $k$-regular ordered two-stage graph on a vertex set $V$. If $F(L)=1$, then $R(L)\le2$. Consequently the graph counterexamples in Theorem \ref{thm:graph} occur at the smallest possible value of $F(L)$ for which the inequality $R(L)\le2F(L)$ can fail. Likewise, in the partition formulation, the Bene\v{s} implication is true for $n=2$, so Corollary \ref{cor:partition} starts at the first possible exponent.
\end{proposition}

\begin{proof}
If $F(L)=1$, then every source is adjacent to every target. Since $L$ is simple and $k$-regular, this means $k=|V|$ and $L$ is the complete ordered bipartite graph from one copy of $V$ to the next. Let $M$ be any $k$-regular mask. By Lemma \ref{lem:konig}, decompose $M$ into $k$ perfect matchings and index the middle vertices of $L^2$ as $z_1,\ldots,z_k$. Route the request $s\to t$ belonging to the $c$th perfect matching as $s\to z_c\to t$. For fixed $c$, the first-stage edges are distinct because the matching uses each source once, and the second-stage edges are distinct because it uses each target once; for different colors the middle vertices are different. Thus $L^2$ is rearrangeable.

For the partition statement, put $S_i=S(\varphi^i(\mathbf u))$. Let $\mathbf u$ be uniform, let $\varphi$ satisfy $\mathbf u\wedge\varphi(\mathbf u)=\mathbf0$, and assume $(\varphi S_0)^2$ is transitive. Since $(\varphi S_0)^2=S_1S_2\varphi^2$, this is equivalent, after multiplying by the fixed permutation $\varphi^{-2}$ and conjugating by $\varphi^{-1}$, to transitivity of $S_0S_1$. By the one-path equivalence proved above, the graph between $\mathbf u$ and $\varphi(\mathbf u)$ has $F=1$. The preceding paragraph makes the two-fold concatenation rearrangeable, and Lemma \ref{lem:routing} translates this into completeness of $\mathbf u,\varphi(\mathbf u),\varphi^2(\mathbf u)$. Equivalently, $S_0S_1S_2=\Sym(E)$, and conjugating the factors gives $(\varphi S_0)^3=\Sym(E)$.
\end{proof}

Let $G$ be a $k$-regular multistage graph from a set $V$ to a second copy of $V$, and let $\mathcal C=(C_1,\ldots,C_k)$ be any partition of $V$. Say that $\mathcal C$ is balanced if $|C_c|=|V|/k$ for every $c$. Say that $G$ is left-universal for $\mathcal C$ if every function $\sigma:V\times\{1,\ldots,k\}\to V$ satisfying $\sigma(v,c)\in C_c$ and $|\sigma^{-1}(z)|=k$ for every $z\in V$ determines a routable mask with requests $v\to\sigma(v,c)$. Say that $G$ is right-universal for $\mathcal C$ if every such function $\tau$ determines a routable mask with requests $\tau(v,c)\to v$.

\begin{theorem}[Balanced middle criterion]\label{thm:middle}
Let $L$ be a simple $k$-regular ordered two-stage graph on a vertex set $V$, and suppose $k$ divides $|V|$. If, for some $m$, the graph $L^m$ is both left-universal and right-universal for a balanced partition $\mathcal C=(C_1,\ldots,C_k)$ of $V$, then $L^{2m}$ is rearrangeable. Consequently $R(L)\le 2m$.
\end{theorem}

\begin{proof}
Let $M$ be any $k$-regular mask from the sources to the targets of $L^{2m}$. By Lemma \ref{lem:konig}, color the requests of $M$ with colors $1,\ldots,k$ so that each source and each target is incident with exactly one request of each color. For each color $c$, the color-$c$ requests are $|V|$ in number, while the vertices of $C_c$ provide $k|C_c|=|V|$ middle slots. Choose a bijection between the color-$c$ requests and these slots. If a request $e$ has source $s(e)$, target $t(e)$, color $c(e)$, and assigned middle vertex $z(e)\in C_{c(e)}$, then each middle vertex is assigned exactly $k$ requests.

Define $\sigma(s,c)=z(e)$, where $e$ is the unique color-$c$ request leaving $s$, and define $\tau(t,c)=z(e)$, where $e$ is the unique color-$c$ request entering $t$. The chosen slots make both functions admissible. Since $L^m$ is left-universal, the first half routes all requests $s(e)\to z(e)$ edge-disjointly. Since $L^m$ is right-universal, the second half routes all requests $z(e)\to t(e)$ edge-disjointly. Concatenating corresponding half-paths gives an edge-disjoint routing of $M$ through $L^{2m}$. Since $M$ was arbitrary, $L^{2m}$ is rearrangeable.
\end{proof}

\begin{remark}
The counterexamples show why external connectivity cannot replace the left- and right-universality hypotheses in Theorem \ref{thm:middle}. External connectivity provides one path between each ordered pair, but it gives no guarantee that the many paths required by the slot assignments can be chosen edge-disjointly. Corollary \ref{cor:sharpdoubling} and Theorem \ref{thm:exactmiddle}, rather than Theorem \ref{thm:middle}, give the logically minimal middle-layer statements.
\end{remark}

\begin{proposition}[Balance cannot be omitted]\label{prop:balance}
If the word ``balanced'' is removed from Theorem \ref{thm:middle}, while the definitions of left- and right-universality are kept unchanged, the resulting statement is false.
\end{proposition}

\begin{proof}
Take $N\ge5$ and $L=L_N$. Let $k=N$ and $|V|=2N$. Choose an unbalanced partition $\mathcal C=(C_1,\ldots,C_N)$ of $V$ into $N$ nonempty parts, for instance with sizes $1,3,2,\ldots,2$. For a function $\sigma:V\times\{1,\ldots,N\}\to V$ satisfying $\sigma(v,c)\in C_c$ and $|\sigma^{-1}(z)|=N$ for every $z$, one would have $|V|=\sum_{z\in C_c}|\sigma^{-1}(z)|=N|C_c|$ for each $c$. This forces $|C_c|=2$ for all $c$, contrary to the choice of $\mathcal C$. Hence no admissible $\sigma$ exists, and the corresponding left- and right-universality hypotheses are vacuous. Nevertheless $L_N^4$ is not rearrangeable, since $R(L_N)\ge N>4$ by Theorem \ref{thm:graph}. Thus the unbalanced version of the criterion would give a false conclusion.
\end{proof}

\begin{remark}
This proposition does not assert that the balanced-middle hypotheses are necessary for rearrangeability. They are not meant to be necessary. What is necessary and sufficient is the exact middle factorization in Theorem \ref{thm:exactmiddle}; balanced universality is a clean sufficient condition that implies it.
\end{remark}

\section{Length-three shuffle-exchange}\label{sec:se}

Let $A$ be an alphabet of size $k$, and let $A^n$ be the set of words of length $n$. In this section permutations act on the right. Thus $x^{\alpha\beta}=(x^\alpha)^\beta$. The shuffle is $(x_1x_2\cdots x_n)^\sigma=x_2\cdots x_nx_1$, and $\calG$ is the group of permutations preserving the first $n-1$ letters. One factor $\sigma\calG$ sends $x_1\cdots x_n$ to $x_2\cdots x_n*$, where the last letter is changed by an arbitrary permutation depending on $x_2\cdots x_n$. This is the standard directed shuffle-exchange routing convention. Let $d(k,n)$ be the least $d\ge2$ such that $\Sym(A^n)=(\sigma\calG)^d$ in this convention.

For $\pi\in\Sym(A^3)$ and $d\ge3$, a $d$-step route for $\pi$ is an assignment to each $x=x_1x_2x_3$ of a middle word $z(x)=z_1\cdots z_{d-3}$ such that, putting $y=\pi(x)$ and $S(x)=x_1x_2x_3z_1\cdots z_{d-3}y_1y_2y_3$, the map $x\mapsto S(x)_{t+1}S(x)_{t+2}S(x)_{t+3}$ is a bijection $A^3\to A^3$ for every $1\le t\le d-1$.

\begin{lemma}\label{lem:route-product}
For $d\ge3$, a permutation $\pi\in\Sym(A^3)$ belongs to $(\sigma\calG)^d$ if and only if it admits a $d$-step route.
\end{lemma}

\begin{proof}
A product of $d$ shuffle-exchange factors gives, for each input packet, a sequence of $d+1$ states in which each state is obtained from the preceding one by deleting the first letter and appending a new letter. Reading the appended letters before the final output gives the middle word above, and every layer is bijective because every factor is a permutation of $A^3$.

Conversely, suppose the sliding-window bijections are given. Between two consecutive layers the transition has the form $abc\mapsto bc*$. For each fixed $bc$, the $k$ states $abc$ in the earlier layer and the $k$ states $bc*$ in the later layer are both present exactly once, so the transition on that fiber is a permutation of $A$. These fiber permutations define the required element of $\calG$ after the shuffle at each layer.
\end{proof}

\begin{proposition}\label{prop:upper-six}
For every $k\ge2$ and every $\pi\in\Sym(A^3)$, the permutation $\pi$ has a six-step route. Hence $d(k,3)\le6$.
\end{proposition}

\begin{proof}
Write an input as $abc$ and its output as $\pi(abc)=def$. Form the bipartite multigraph $M$ with left vertices $L_{bc}$, right vertices $R_{de}$, and one edge for each packet $abc\mapsto def$, joining $L_{bc}$ to $R_{de}$. The graph is $k$-regular: for fixed $bc$ there are $k$ choices of $a$, and for fixed $de$ there are $k$ choices of $f$.

By Lemma \ref{lem:konig}, decompose $M$ into perfect matchings $M_w$, $w\in A$. Identify $A$ with $\Z/k\Z$ for this paragraph. Fix $w$ and collapse $M_w$ to a bipartite multigraph $N_w$ with left vertices $c\in A$ and right vertices $d\in A$ by sending an edge $L_{bc}R_{de}$ to an edge $cd$. Since $M_w$ covers every $L_{bc}$ and every $R_{de}$ once, $N_w$ is again $k$-regular. Decompose $N_w$ into perfect matchings indexed by $v\in A$.

Label a packet edge lying in $M_w$ and in the $v$-matching of $N_w$ by the middle word $uvw$, where $u=w+c$. The five nontrivial windows are $bcu$, $cuv$, $uvw$, $vwd$, and $wde$. They are bijective: $bcu$ determines $w=u-c$ and then the unique edge of $M_w$ at $L_{bc}$; $cuv$ and $uvw$ determine $w$ and $c$, then use the unique edge of the $v$-matching of $N_w$ at $c$; $vwd$ uses the unique edge of that matching at $d$; and $wde$ uses the unique edge of $M_w$ at $R_{de}$. Lemma \ref{lem:route-product} gives the required factorization.
\end{proof}

For completeness we exclude shorter lengths.

\begin{lemma}\label{lem:lower-five}
For every $k\ge2$, not every permutation of $A^3$ is routable in at most four steps. Thus $d(k,3)\ge5$.
\end{lemma}

\begin{proof}
For two steps the first letter of the final state is forced to be the third letter of the input. For three steps, choose a permutation which maps the $k$ inputs $a00$, $a\in A$, to outputs whose first letter is $0$; the first intermediate layer would contain the same state $000$ for all these packets. For four steps, choose a permutation which maps all $k^2$ inputs $ab0$ to outputs whose first letter is $0$. With one middle letter $z$, the central window is $0z0$, which has only $k$ possible values for these $k^2$ packets.
\end{proof}

It remains to prove that five steps do not always suffice when $k\ge3$. Put $A=\Z/k\Z$. A five-step route for $abc\mapsto def$ assigns a middle word $uv$, and the four nontrivial windows are $bcu$, $cuv$, $uvd$, and $vde$. Equivalently, in the multigraph on vertices $L_{bc}$ and $R_{de}$, every edge must be colored by a pair $(u,v)\in A^2$ so that at each $L_{bc}$ the $u$-values are all distinct, in each fixed input block $c$ all $k^2$ colors occur once, in each fixed output block $d$ all $k^2$ colors occur once, and at each $R_{de}$ the $v$-values are all distinct.

For a left vertex $L$, let $F_L\subseteq A^2$ be the set of colors on its incident edges; define $H_R$ similarly for a right vertex. If there are $k$ parallel edges from $L$ to $R$, then $F_L=H_R$. This common set is a perfect matching in $A\times A$: it has one point with each first coordinate and one point with each second coordinate. We shall use the elementary fact that two perfect matchings in $A\times A$ sharing $k-1$ points are equal, since the last unused row and column force the final point.

Construct a $k$-regular bipartite multigraph $M_k$ as follows. First, for every $c\in A$ and $0\le b\le k-2$, put $k$ parallel edges $L_{bc}\to R_{cb}$, and for every $c\in A$ put $k$ parallel edges $L_{k-1,c}\to R_{c+1,k-1}$, with $c+1$ modulo $k$. Then replace the full edge $L_{k-1,0}\to R_{1,k-1}$ and the full edge $L_{0,2}\to R_{2,0}$ by
\[
(k-1)(L_{k-1,0}\to R_{1,k-1})+(L_{k-1,0}\to R_{2,0})+(k-1)(L_{0,2}\to R_{2,0})+(L_{0,2}\to R_{1,k-1}).
\]
Every vertex has degree $k$: before the replacement each vertex lay in exactly one full $k$-fold edge, and the replacement restores degree $k$ at the four affected vertices.

\begin{proposition}\label{prop:obstruction}
The multigraph $M_k$ has no five-step coloring.
\end{proposition}

\begin{proof}
Suppose such a coloring exists. For the undeleted full edges write $B_{c,b}=F_{L_{bc}}=H_{R_{cb}}$ for $0\le b\le k-2$, and $X_c=F_{L_{k-1,c}}=H_{R_{c+1,k-1}}$ for $c\ne0$, omitting only $B_{2,0}$ and $X_0$. Each displayed set is a perfect matching in $A\times A$.

For every $c\notin\{0,1,2\}$, the input block $c$ and the output block $c$ have the same parts $B_{c,0},\ldots,B_{c,k-2}$; their remaining parts are $X_c$ and $X_{c-1}$, respectively. Hence $X_c=X_{c-1}$, and therefore $X_{k-1}=X_2$, with the assertion tautological when $k=3$.

Let $X_L=L_{k-1,0}$, $X_R=R_{1,k-1}$, $Y_L=L_{0,2}$, and $Y_R=R_{2,0}$. Comparing input block $0$ with output block $0$ gives $F_{X_L}=X_{k-1}=X_2$, and comparing input block $1$ with output block $1$ gives $H_{X_R}=X_1$. Since there are $k-1$ parallel edges from $X_L$ to $X_R$, the matchings $X_2$ and $X_1$ share at least $k-1$ colors. Thus $X_1=X_2$; call this matching $X$.

Now compare input block $2$ and output block $2$. Both partitions contain $X$ and the common parts $B_{2,b}$ for $1\le b\le k-2$. Their remaining parts are $F_{Y_L}$ and $H_{Y_R}$, so these are equal; call the common set $Y$. Since $X$ and $Y$ occur as two parts of the same partition of $A^2$, they are disjoint. But the single edge $X_L\to Y_R$ would need a color in $F_{X_L}\cap H_{Y_R}=X\cap Y$, a contradiction.
\end{proof}

Because $M_k$ is $k$-regular, it lifts to an actual permutation of $A^3$: enumerate the $k$ edges incident with each $L_{bc}$ by $a\in A$, enumerate the $k$ edges incident with each $R_{de}$ by $f\in A$, and define $\pi(a,b,c)=(d,e,f)$ when the corresponding edge joins $L_{bc}$ to $R_{de}$. A five-step route for this $\pi$ would give a forbidden coloring of $M_k$.

\begin{theorem}\label{thm:mainse}
For every $k\ge3$, one has $d(k,3)=6$.
\end{theorem}

\begin{proof}
Lemma \ref{lem:lower-five} gives $d(k,3)\ge5$. Proposition \ref{prop:obstruction} and the preceding lift show that $d(k,3)$ is not $5$, so $d(k,3)\ge6$. Proposition \ref{prop:upper-six} gives $d(k,3)\le6$.
\end{proof}

Thus the shuffle-exchange conjecture $d(k,n)=2n-1$ fails for every $k\ge3$ already at word length three. For comparison, the known binary value is $d(2,3)=5$; see \cite{OPGshuffle,RaghavendraVarma87,LinialTarsi89,NgoDu01}.

\section{Literature and applications}

The explicit graph family should be read with a cautious priority convention. We have not found the graphs $L_N$, nor the conclusion $F(L_N)=2$ with unbounded $R(L_N)$, in the accessible records cited here. The Open Problem Garden pages state the partition and graph formulations as open, and the graph page formulates the question as whether external connectivity of $L^m$ forces rearrangeability of $L^{2m}$ for simple regular ordered two-stage graphs \cite{OPGpart,OPGgraph}. Hwang's 1989 abstract is important context: it says that Bene\v{s} published the conjecture in a more general setting and that a mathematical abstraction also shows that no number of stages guarantees a nonblocking network \cite{Hwang89}. Since that statement is close in spirit but not identical to the rearrangeability assertion treated here, this paper makes no priority claim beyond giving explicit counterexamples to the stated Open Problem Garden and UnsolvedMath formulations.

The classical Bene\v{s} network itself is not contradicted. Bene\v{s}' original book and later group-calculation proof concern structured connecting networks whose middle symmetry is much stronger than arbitrary external connectivity \cite{Benes65,Benes75}. Theorem \ref{thm:graph} instead shows that the broad graph-theoretic and partition-stabilizer extensions need additional assumptions. Proposition \ref{prop:cut}, Theorem \ref{thm:exactmiddle}, and Theorem \ref{thm:middle} indicate the type of assumptions that can work: a useful replacement must control balanced traffic through cuts or through a prescribed middle layer, not merely guarantee one path between each ordered pair.

The shuffle-exchange problem is more structured. Stone's perfect-shuffle work is an early source for binary networks \cite{Stone71}. Before Theorem \ref{thm:mainse}, the Open Problem Garden page recorded the exact cases $d(k,2)=3$, $d(2,3)=5$, and $d(2,4)=7$, the universal lower bound $d(k,n)\ge2n-1$, the recurrence $d^{(*)}(k,n)\le d^{(*)}(k,r)+3(n-r)$, and hence $d(k,n)\le3n-3$ for $k\ge3$ \cite{OPGshuffle}. It also states that the graph-theoretic parameter $r(k,n)$ equals $d(k,n)$ \cite{OPGsegraph}. Work on the conjecture and its constructive versions includes Raghavendra--Varma, Varma--Raghavendra, Linial--Tarsi, Kim--Raghavendra, Ngo--Du, Bashirov, \c{C}am, Bao--Hwang--Li, and Dai--Shen \cite{RaghavendraVarma87,VarmaRaghavendra88,LinialTarsi89,KimRaghavendra91,Raghavendra95,NgoDu01,Bashirov01,Cam03,BaoHwangLi06,DaiShen08}. The Open Problem Garden page records flawed or refuted proposed proofs in this line, and Bao--Hwang--Li explicitly refuted the claimed binary proof of \c{C}am \cite{BaoHwangLi06,OPGshuffle}.

Theorem \ref{thm:mainse} changes the first unsettled non-binary case from $5\le d(3,3)\le6$ to the exact value $d(3,3)=6$, and proves the same value for every alphabet size $k\ge3$. Since the known upper bound for $k\ge3$ is $3n-3$, a natural replacement question is whether $d(k,n)=3n-3$ for all $k\ge3$ and $n\ge2$. This paper proves only the base case $n=3$ of that possible formula. The first new cases include $d(3,4)$ and $d(4,4)$; for the latter the general bounds give $7\le d(4,4)\le9$. The proof here is intrinsically two-coordinate: a full bundle at length three pins a perfect matching in $A^2$, and two such matchings sharing $k-1$ points are equal; for longer words, extra middle coordinates can absorb this rigidity unless a new pinning mechanism is introduced.

Later work on banyan-type networks gives sufficient conditions, SAT models, and empirical searches for rearrangeable configurations \cite{LiTan09,Ohta21,Ohta23}. These papers address highly constrained network families, so they do not remove the bottleneck obstruction in the fully general class considered here. The length-three shuffle-exchange obstruction is also separate from the bottleneck family: it lives inside the homogeneous shuffle-exchange network and is caused by incompatible perfect matchings rather than by a visible source-target cut.

The connection with machine learning is indirect but concrete. In expert-parallel mixture-of-experts models, tokens are routed to the devices that hold their selected experts, which induces all-to-all communication patterns across accelerators \cite{GShard,Switch22,DeepSpeedMoE}. The results here give worst-case combinatorial warnings for such systems: small diameter, transitivity, or pairwise reachability is not enough to support balanced all-to-all traffic; one also needs cut capacity, middle-layer factorization, or a well-linkedness condition strong enough to route many edge-disjoint paths at once. The paper proves no new training algorithm, but it supplies compact obstruction tests for topological bottlenecks and for local shuffle-exchange rigidity, both of which can inform interconnect design, expert placement, and routing stress tests.

\end{document}